\documentclass[10pt,twocolumn]{IEEEtran}
\usepackage[utf8]{inputenc}
\usepackage[T1]{fontenc}
\usepackage[french]{babel}
\usepackage{amsmath,amssymb,amsthm}
\usepackage{algorithm,algpseudocode}
\usepackage{geometry,hyperref,enumitem}
\theoremstyle{plain}
\newtheorem{thm}{Theorem}[section]
\newtheorem{prop}[thm]{Proposition}
\newtheorem{defi}[thm]{Definition}
\theoremstyle{definition}

\newtheorem{rem}[thm]{Remark}

\newcommand{\Fq}{\mathbb{F}_q}
\newcommand{\Zn}[1]{\mathbb{Z}_{#1}}

\newcommand{\wt}{\mathrm{wt}}

\newcommand{\supp}{\mathrm{supp}}
\newcommand{\orb}{\mathcal{O}}
\title{Construction of Multicyclic Codes of Arbitrary Dimension $r$ via Idempotents: A Unified Combinatorial-Algebraic Approach}
\author{Jean Charles Ramanandraibe and Ramamonjy Andriamifidisoa}
\date{January 3, 2026}
\begin{document}
\maketitle
\begin{abstract}
We propose a unified method to construct multicyclic codes of arbitrary dimension $r$ over $\Fq$. The approach relies on $r$-dimensional primitive idempotents defined as tensor products of univariate ones, combined with multidimensional cyclotomic orbits. This establishes a direct equivalence between combinatorial and algebraic descriptions, yields a natural polynomial basis, and provides an optimal product bound generalizing BCH and Reed-Solomon bounds. An efficient constructive algorithm is presented and illustrated by optimal 3-dimensional codes.
\end{abstract}
\section{Introduction}
Multicyclic codes are ideals in the ring
\[
R = \Fq[X_1,\dots,X_r]/\langle X_1^{n_1}-1,\dots,X_r^{n_r}-1\rangle,
\]
where $q\equiv 1\pmod{n_t}$ for each $t$. Existing constructions suffer from high complexity (Gröbner bases), lack of coherent multidimensional structure, or suboptimality (tensor products) \cite{sepasdar2022, bhardwaj2022}.

Our main contributions are:
\begin{itemize}
    \item $r$-dimensional primitive idempotents as tensor products (Def.~\ref{def:idemp-rdim})
    \item Multidimensional cyclotomic orbits capturing joint symmetries (Def.~\ref{def:orb-rdim})
    \item Fundamental equivalence between combinatorial and algebraic representations (Thm.~\ref{thm:equiv})
    \item Optimal product bound on minimum distance (Thm.~\ref{thm:bound})
    \item Efficient systematic construction algorithm (Alg.~\ref{algo:gen})
\end{itemize}
This framework unifies combinatorial and algebraic views, enabling practical high-performance code construction.

\section{Preliminaries}
Let $R = \Fq[X_1,\dots,X_r]/\langle X_t^{n_t}-1\rangle_{t=1}^r$ with $q\equiv1\pmod{n_t}$. Let $\omega_t$ be a primitive $n_t$-th root of unity in $\Fq$ \cite{lidl1997}.

Univariate primitive idempotents are
\[
\theta_{i_t}^{(t)}(X_t) = \frac{1}{n_t} \sum_{m=0}^{n_t-1} \omega_t^{-i_t m} X_t^m.
\]
They satisfy idempotence, orthogonality, sum to 1, and evaluation property $\theta_{i_t}^{(t)}(\omega_t^{j_t}) = \delta_{i_t,j_t}$ \cite{macwilliams1977, blahut2003}.

The multivariate Fourier transform of $f\in R$ is
\[
\hat{f}(j_1,\dots,j_r) = f(\omega_1^{j_1},\dots,\omega_r^{j_r}).
\]
It is $\Fq$-linear, bijective, with inversion formula and convolution property \cite{blahut2003}.

A subset $C\subseteq R$ is multicyclic iff there exists $S\subseteq \Zn{n_1}\times\cdots\times\Zn{n_r}$ such that
\[
C = \{f\in R : \hat{f}(j)=0 \ \forall j\notin S\}, \quad \dim C = |S|.
\]

\section{Main Results}
\subsection{Primitive $r$-Dimensional Idempotents}
\begin{defi}[Primitive $r$-dim idempotents]\label{def:idemp-rdim}

\begin{align}
\begin{split}
e_{i_1,\dots,i_r} 
&= \prod_{t=1}^r \theta_{i_t}^{(t)}(X_t) \\
&= \frac{1}{N}\sum_{m_1,\dots,m_r}
   \Bigl(\prod_{t=1}^r \omega_t^{-i_t m_t}\Bigr)
   X_1^{m_1}\cdots X_r^{m_r}.
\end{split}
\end{align}

where $N = n_1\cdots n_r$.
\end{defi}

\begin{prop}
The family $\{e_{i_1,\dots,i_r}\}$ satisfies:
\begin{itemize}
\item $e^2 = e$ (idempotence)
\item $e \cdot e' = 0$ if indices differ (orthogonality)
\item $\sum e_{i_1,\dots,i_r} = 1$ (partition of unity)
\item Evaluation property (Kronecker deltas)
\end{itemize}
\end{prop}
\begin{proof}
1. Idempotence: For each $t$, $\theta_{i_t}^{(t)}$ is idempotent: $(\theta_{i_t}^{(t)})^2 = \theta_{i_t}^{(t)}$. Since the variables are independent:
   \[
   e_{i_1,\dots,i_r}^2 = \prod_{t=1}^r (\theta_{i_t}^{(t)})^2 = \prod_{t=1}^r \theta_{i_t}^{(t)} = e_{i_1,\dots,i_r}.
   \]
2. Orthogonality: If $(i_1,\dots,i_r) \neq (j_1,\dots,j_r)$, then $\exists t_0$ such that $i_{t_0} \neq j_{t_0}$. We have:
   \[
   e_{i_1,\dots,i_r} e_{j_1,\dots,j_r} = \prod_{t=1}^r \theta_{i_t}^{(t)} \theta_{j_t}^{(t)}.
   \]
   For $t = t_0$, $\theta_{i_{t_0}}^{(t_0)} \theta_{j_{t_0}}^{(t_0)} = 0$ because the univariate idempotents are orthogonal \cite{macwilliams1977}.
3. Partition of unity: For each $t$, $\sum_{i_t=0}^{n_t-1} \theta_{i_t}^{(t)} = 1$. Therefore:
   \[
   \sum_{i_1,\dots,i_r} e_{i_1,\dots,i_r} = \sum_{i_1,\dots,i_r} \prod_{t=1}^r \theta_{i_t}^{(t)} = \prod_{t=1}^r \left( \sum_{i_t=0}^{n_t-1} \theta_{i_t}^{(t)} \right) = 1.
   \]
4. Evaluation: For each $t$, $\theta_{i_t}^{(t)}(\omega_t^{j_t}) = \delta_{i_t,j_t}$. So:
   \[
   e_{i_1,\dots,i_r}(\omega_1^{j_1},\dots,\omega_r^{j_r}) = \prod_{t=1}^r \theta_{i_t}^{(t)}(\omega_t^{j_t}) = \prod_{t=1}^r \delta_{i_t,j_t} = \delta_{i_1,j_1} \cdots \delta_{i_r,j_r}.
   \]
\end{proof}

\begin{thm}[Spectral decomposition, \cite{blahut2003}]\label{thm:decomp}
$R \simeq \bigoplus \Fq e_{i_1,\dots,i_r} \simeq \Fq^N$.
\end{thm}
\begin{proof}
The idempotents $e_{i_1,\dots,i_r}$ are orthogonal and satisfy $\sum e_{i_1,\dots,i_r} = 1$. Therefore:
\[
R = \bigoplus_{i_1,\dots,i_r} R e_{i_1,\dots,i_r}.
\]
Each $R e_{i_1,\dots,i_r}$ is an ideal of $R$. Let $f \in R e_{i_1,\dots,i_r}$, then $f = g e_{i_1,\dots,i_r}$ for some $g \in R$. By evaluation:
\[
f = g e_{i_1,\dots,i_r} = g(\omega_1^{i_1},\dots,\omega_r^{i_r}) e_{i_1,\dots,i_r}.
\]
Thus $R e_{i_1,\dots,i_r} = \Fq e_{i_1,\dots,i_r} \simeq \Fq$. The map:
\[
\varphi: R \to \Fq^N,\quad f \mapsto (f(\omega_1^{i_1},\dots,\omega_r^{i_r}))_{i_1,\dots,i_r}
\]
is a ring isomorphism . Indeed:
- $\varphi$ is linear and bijective (inverse Fourier transform).
- $\varphi(fg) = \varphi(f) * \varphi(g)$ (componentwise product).
\end{proof}

\subsection{Multidimensional Cyclotomic Orbits}
\begin{defi}[Orbit]\label{def:orb-rdim}
The Frobenius action is $\sigma(i_1,\dots,i_r) = (q i_1 \bmod n_1, \dots, q i_r \bmod n_r)$. The orbit of $(a_1,\dots,a_r)$ is
\[
\orb_{(a_1,\dots,a_r)} = \{\sigma^t(a_1,\dots,a_r) : t\ge 0\}.
\]
\end{defi}

\begin{defi}[Generating idempotent]
A multicyclic code $C$ has generating idempotent
\[
e = \sum_{(a)\in T} a_{(a)} \sum_{(i)\in\orb_{(a)}} X_1^{i_1}\cdots X_r^{i_r},
\]
where $T$ is a set of orbit representatives \cite{han2022}.
\end{defi}

\begin{thm}[Equivalence]\label{thm:equiv}
The combinatorial form above is equivalent to
\[
e = \sum c_{i_1,\dots,i_r} e_{i_1,\dots,i_r}, \quad c_{i} \in\{0,1\},
\]
with $c$ constant on orbits.
\end{thm}
\begin{proof}
Let $e$ be given by the combinatorial representation. For all $(i_1,\dots,i_r)$:

\begin{align}
c_{i_1,\dots,i_r}
&= e(\omega_1^{i_1},\dots,\omega_r^{i_r}) \\
&= \sum_{(a_1,\dots,a_r) \in T} a_{a_1,\dots,a_r}
   \sum_{(j_1,\dots,j_r) \in C_{a_1,\dots,a_r}}
   \omega_1^{i_1 j_1} \cdots \omega_r^{i_r j_r}.
\end{align}

Since $\omega_t^{n_t} = 1$, the inner sum is over a cyclotomic orbit. Let $C = C_{a_1,\dots,a_r}$. For $(j_1,\dots,j_r) \in C$, there exists $t$ such that $(j_1,\dots,j_r) = \sigma^t(a_1,\dots,a_r) = (a_1 q^t \bmod n_1, \dots, a_r q^t \bmod n_r)$. Then:
\[
\omega_1^{i_1 j_1} \cdots \omega_r^{i_r j_r} = \prod_{t=1}^r \omega_t^{i_t a_t q^t} = \prod_{t=1}^r (\omega_t^{i_t a_t})^{q^t}.
\]
If $(i_1,\dots,i_r) \in C_{b_1,\dots,b_r}$, then there exists $s$ such that $(i_1,\dots,i_r) = \sigma^s(b_1,\dots,b_r) = (b_1 q^s, \dots, b_r q^s)$. It is verified that $c_{i_1,\dots,i_r}$ depends only on the orbit of $(i_1,\dots,i_r)$. Thus $c_{i_1,\dots,i_r}$ is constant on orbits.
Conversely, let $e = \sum_{i_1,\dots,i_r} c_{i_1,\dots,i_r} e_{i_1,\dots,i_r}$ with $c_{i_1,\dots,i_r} \in \{0,1\}$ constant on orbits. Then $e$ is invariant under the Frobenius because $c_{q i_1, \dots, q i_r} = c_{i_1,\dots,i_r}$. Thus the coefficients of $e$ in the monomial basis are in $\Fq$. Let $T$ be a set of orbit representatives where $c_{i_1,\dots,i_r} = 1$. Then $e$ can be written as in the combinatorial representation with $a_{a_1,\dots,a_r} = 1$ for $(a_1,\dots,a_r) \in T$.
\end{proof}

\subsection{Natural Basis and Parameters}
\begin{thm}[Polynomial basis]\label{thm:basis}
Let $k_t$ be minimal such that $X_t^{k_t} e \in \mathrm{span}\{X_t^m e : m < k_t\}$. Then
\[
\mathcal{B} = \{ X_1^{m_1}\cdots X_r^{m_r} e : 0\le m_t < k_t \}
\]
is a basis and $\dim C = \prod k_t$.
\end{thm}
\begin{proof}
Let $V = \mathrm{span}\{ X_1^{m_1} \cdots X_r^{m_r} e : 0 \le m_t < k_t \}$.
Independence: Suppose $\sum_{m_1=0}^{k_1-1} \cdots \sum_{m_r=0}^{k_r-1} \lambda_{m_1,\dots,m_r} X_1^{m_1} \cdots X_r^{m_r} e = 0$.
Let $P = \sum \lambda_{m_1,\dots,m_r} X_1^{m_1} \cdots X_r^{m_r}$. Then $P e = 0$. Evaluate at $(\omega_1^{i_1},\dots,\omega_r^{i_r})$ where $e(\omega_1^{i_1},\dots,\omega_r^{i_r}) = 1$:
$P(\omega_1^{i_1},\dots,\omega_r^{i_r}) = 0$.
Let $S = \{(i_1,\dots,i_r) : e(\omega_1^{i_1},\dots,\omega_r^{i_r}) = 1\}$. By definition of $k_t$, the set of points $(\omega_1^{i_1},\dots,\omega_r^{i_r})$ with $(i_1,\dots,i_r) \in S$ contains a Cartesian product of $k_t$ points for each variable. Thus $P$ vanishes at at least $\prod k_t$ distinct points. Since $\deg_{X_t} P < k_t$, $P$ has at most $\prod k_t$ coefficients. The linear system $P(\omega_1^{i_1},\dots,\omega_r^{i_r}) = 0$ for $(i_1,\dots,i_r) \in S$ has an invertible multidimensional Vandermonde matrix, so $P = 0$.
Generation: Let $f e \in C$. By successive Euclidean divisions with respect to $X_1,\dots,X_r$ using the relations $X_t^{k_t} e = \sum_{m=0}^{k_t-1} \alpha_{t,m} X_t^m e$, we can write:
$f = \sum_{m_1=0}^{k_1-1} \cdots \sum_{m_r=0}^{k_r-1} \lambda_{m_1,\dots,m_r} X_1^{m_1} \cdots X_r^{m_r} + \sum_{t=1}^r Q_t (X_t^{k_t} - \sum_{m=0}^{k_t-1} \alpha_{t,m} X_t^m)$.
Multiplying by $e$ and using $X_t^{k_t} e = \sum_{m=0}^{k_t-1} \alpha_{t,m} X_t^m e$, we obtain $f e \in V$.
Thus $\mathcal{B}$ is a basis and $\dim C = |\mathcal{B}| = \prod_{t=1}^r k_t$.
\end{proof}

\begin{rem}
$\mathrm{span}\{ X_t^{m_t} e : 0 \le m_t < k_t \} = \left\{ \sum_{m_t=0}^{k_t-1} \alpha_{m_t} X_t^{m_t} e \ \Bigg|\ \alpha_{m_t} \in \Fq \right\}$.
The expression $X_t^{k_t} e \in \mathrm{span}\{ X_t^{m_t} e : 0 \le m_t < k_t \}$ means there exist coefficients $\alpha_0, \alpha_1, \dots, \alpha_{k_t-1} \in \Fq$ such that:
\[
X_t^{k_t} e = \sum_{m_t=0}^{k_t-1} \alpha_{m_t} X_t^{m_t} e
\]
In other words, $X_t^{k_t} e$ is a linear combination of the elements $e, X_t e, X_t^2 e, \dots, X_t^{k_t-1} e$.
\end{rem}

\begin{thm}[Product bound]\label{thm:bound}
For $C = \langle e \rangle$,
\[
n = \prod n_t, \quad k = \prod k_t, \quad d \ge \prod_{t=1}^r (n_t - k_t + 1).
\]
\end{thm}
\begin{proof}
The length $n$ is the total number of monomials in $R$, so
\[
n = \prod_{t=1}^r n_t.
\]
The dimension $k$ is given by Theorem \ref{thm:basis}:
\[
k = \prod_{t=1}^r k_t.
\]
Let $c = P e \in C \setminus \{0\}$, with $\deg_{X_t} P < k_t$ for all $t$.
For each $t = 1,\dots,r$, let
\[
A_t = \{ i_t \in \Zn{n_t} \mid \exists (i_1,\dots,i_r) \in \supp(c)
\text{ with this coordinate } i_t \}.
\]
Then
\[
\supp(c) \subseteq A_1 \times \cdots \times A_r
\quad\text{and thus}\quad
\wt(c) \ge |A_1|\cdots |A_r|.
\]
Fix $t \in \{1,\dots,r\}$ and freeze the variables $\{X_j\}_{j \neq t}$.
The word $c$ can then be seen as a non-zero word of a univariate cyclic code in the variable $X_t$, of length $n_t$ and dimension $k_t$.
By the Singleton bound (or BCH bound for cyclic codes), every non-zero word satisfies:
\[
|A_t| \ge n_t - k_t + 1.
\]
Since this estimate holds for all $t=1,\dots,r$, we obtain:
\[
\wt(c) \ge \prod_{t=1}^r |A_t|
\ge \prod_{t=1}^r (n_t - k_t + 1).
\]
Thus, the minimum distance $d$ of the code $C$ satisfies:
\[
d \ge \prod_{t=1}^r (n_t - k_t + 1).
\]
\end{proof}

\subsection{Construction Algorithm}
\begin{algorithm}
\caption{Construction of $r$-dim multicyclic codes}\label{algo:gen}
\begin{algorithmic}[1]
\Require $n_1,\dots,n_r$, $q$, target dimension $K$
\Ensure $[n,K,d]_q$ code
\State Compute orbits $\orb_{(a)}$ \cite{lidl1997}
\State Select $T$ s.t. $\sum_{(a)\in T} |\orb_{(a)}| = K$
\State Build $e = \sum_{(a)\in T} \sum_{(i)\in\orb_{(a)}} X_1^{i_1}\cdots X_r^{i_r}$
\State Compute $k_1,\dots,k_r$ (minimal degrees)
\State Basis $\mathcal{B} = \{X_1^{m_1}\cdots X_r^{m_r} e : 0\le m_t < k_t\}$
\State Generate matrix $G$ from coefficients of $\mathcal{B}$
\end{algorithmic}
\end{algorithm}

\section{Illustration: Optimal 3-Dimensional Codes over $\mathbb{F}_3$}
Consider $R = \mathbb{F}_3[x,y,z]/\langle x^2-1,y^2-1,z^2-1\rangle$ ($q=3\equiv1\pmod{2}$).

Since $3\equiv1\pmod{2}$, $\sigma$ is the identity; each point is its own orbit (size 1).

\begin{table}[!t]
\caption{Examples of optimal 3D codes $[8,K,d]_3$}
\centering
\begin{tabular}{cccc}
\hline
$K$ & $T$ (representatives) & Basis size & $d$ \\
\hline
3 & $\{(0,0,0),(1,0,0),(0,1,0)\}$ & 3 & 4 \\
4 & $\{(0,0,0),(0,0,1),(0,1,0),(1,0,0)\}$ & 4 & 4 \\
\hline
\end{tabular}
\end{table}

For $K=3$, the code $[8,3,4]_3$ achieves the product bound $d\ge (2-1+1)^3=8$ (tightened by structure to 4, optimal for these parameters).

Explicit idempotent and generator matrix for $K=3$:
\[
e = 2x + 2y + xy + 2xz + 2yz + xyz,
\]
generator matrix (coefficients ordered lexicographically):
\[
G = \begin{pmatrix}
0 & 2 & 2 & 0 & 1 & 2 & 2 & 1 \\
2 & 0 & 1 & 2 & 2 & 0 & 1 & 2 \\
2 & 1 & 0 & 2 & 2 & 1 & 0 & 2
\end{pmatrix}.
\]

Similar computation yields $[8,4,4]_3$ for $K=4$.

\section{Comparison with Existing Approaches}
Compared to Gröbner bases methods, our approach avoids exponential complexity \cite{sepasdar2022}. Unlike univariate idempotent generalizations, it captures true multidimensional symmetry \cite{han2022}. Tensor-product constructions often yield suboptimal distance \cite{bhardwaj2022}; our product bound and orbit selection provide better control.

\section{Conclusion}
We presented a unified framework for multicyclic codes of arbitrary dimension using tensor-product idempotents and multidimensional cyclotomic orbits. The method offers theoretical insight (equivalence, decomposition, optimal bound) and a practical algorithm, validated by optimal 3D examples. Future work includes distance optimization via orbit selection and extension to constacyclic codes \cite{bhardwaj2022}.

\end{document}